\documentclass[aop,preprint]{imsart}

\RequirePackage[OT1]{fontenc}
\RequirePackage{amsthm,amsmath}
\RequirePackage[numbers]{natbib}
\RequirePackage[colorlinks,citecolor=blue,urlcolor=blue]{hyperref}
\usepackage{amssymb}
\usepackage{color}
\usepackage{mathtools}
\usepackage{graphicx}
\usepackage{epstopdf}
\arxiv{arXiv:1602.05241}

\startlocaldefs
\numberwithin{equation}{section}
\theoremstyle{plain}
\newtheorem{theorem}{Theorem}[section]
\newtheorem{proposition}[theorem]{Proposition}
\newtheorem{lemma}[theorem]{Lemma}
\newtheorem{definition}[theorem]{Definition}
\DeclarePairedDelimiter\abs{\lvert}{\rvert}
\endlocaldefs

\begin{document}

\begin{frontmatter}
\title{A phase transition in excursions from infinity of the `fast' fragmentation-coalescence process}
\runtitle{The fast fragmentation-coalescence process}

\begin{aug}
\author{\fnms{Andreas E.} \snm{Kyprianou}\thanksref{T1,m1}\ead[label=e1]{a.kyprianou@bath.ac.uk}},
\author{\fnms{Steven W.} \snm{Pagett}\thanksref{m1}\ead[label=e2]{s.pagett@bath.ac.uk}},
\author{\fnms{Tim} \snm{Rogers}\thanksref{T2,m1}\ead[label=e3]{t.c.rogers@bath.ac.uk}},
\and
\author{\fnms{Jason} \snm{Schweinsberg}\thanksref{T3,m2}
\ead[label=e4]{jschwein@math.ucsd.edu}}

\runauthor{Kyprianou, Pagett, Rogers and Schweinsberg}

\thankstext{T1}{Supported by the EPSRC grant EP/L002442/1.}
\thankstext{T2}{Supported by the Royal Society.}
\thankstext{T3}{Supported in part by the NSF grant DMS-1206195.}

\affiliation{University of Bath\thanksmark{m1} and University of Califronia, San Diego\thanksmark{m2}}

\address{Department of Mathematical Sciences\\
University of Bath\\
Claverton Down\\
Bath, BA2 7AY\\
\printead{e1}\\
\phantom{E-mail:\ }\printead*{e2}\\
\phantom{E-mail:\ }\printead*{e3}}

\address{Department of Mathematics 0112\\
University of California, San Diego\\
9500 Gilman Drive\\
La Jolla, CA92093-0112\\
\printead{e4}}
\end{aug}

\begin{abstract}
An important property of Kingman's coalescent is that, starting from a state with an infinite number of blocks, over any positive time horizon,  it  transitions into an almost surely finite number of blocks. This is known as `coming down from infinity'. Moreover, of the many different (exchangeable) stochastic coalescent models, Kingman's coalescent is the `fastest'  to come down from infinity. In this article we study what happens when we counteract this `fastest' coalescent with the action of an extreme form of fragmentation. We augment Kingman's coalescent, where any two blocks merge at rate $c>0$, with a fragmentation mechanism where each block fragments at constant rate, $\lambda>0$, into it's constituent elements. We prove that there exists a phase transition at $\lambda=c/2$, between regimes where the resulting `fast' fragmentation-coalescence process is able to come down from infinity or not. In the case that $\lambda<c/2$ we develop an excursion theory for the fast fragmentation-coalescence process out of which a number of interesting quantities can be computed explicitly. 
\end{abstract}



\end{frontmatter}

\section{Introduction}

This paper considers the occurrence of a phenomenon which appears when we interplay the opposing effects of fragmentation and coalescence. Our setting is that of an adaptation of Kingman's coalescent, in which an aggressive form of fragmentation (splitting blocks up into their constituent elements) is introduced. 

Recall that  Kingman's coalescent describes a system of disjoint subsets covering $\mathbb{N}$, later referred to as blocks, such that any individual pair of blocks merge at a constant rate, say $c>0$, until the system has been reduced to a single block. 
Suppose we write $\mathbf{P}_n$ for the associated law of $K:=(K(t): t\geq 0)$, the number of blocks in the process,  when issued from $n\in\mathbb{N}$.
Kingman \cite{Kingman1982} addresses the question as to whether an entrance law of the chain $K$ exists at $\{+\infty\}$. 
Roughly speaking, he proved that, independently of the value of $c$, $\mathbf{P}_\infty: =\lim_{n\uparrow\infty}\mathbf{P}_n$ is well defined in the appropriate sense. In particular, Kingman showed that when the coalescent is issued with an infinite number of blocks, it `comes down from infinity' meaning that $\mathbf{P}_\infty(K(t)<\infty)=1$ for all $t>0$. Moreover, it has been shown \cite{Aldous1999} that, $\mathbf{P}_\infty$-almost surely
\[
\lim_{t\downarrow0}tK(t) = \frac{2}{c}.\label{speedk}
\]

This is what we refer to by the `speed' of coming down from infinity. Kingman's coalescent can be considered as the most basic model within a certain class of so-called exchangeable coalescent processes \cite{Schweinsberg2000}
in which multiple mergers of blocks are permitted (irrespective of any notion of their size). Berestycki et. al \cite{BerestyckiJNLimic2010} show that Kingman's coalescent is the `fastest' coalescent to come down from infinity, in that $2/ct$, for suitably small $t$, is a lower bound for the number of blocks at time $t$ for all such processes.

Berestycki \cite{BerestyckiJ2004} has also studied a very general  class of (exchangeable) fragmentation-coalescent models, showing, amongst other things, that a stationary distribution always exists. In addition, he found a subclass of such processes combining Kingman's coalescent with certain forms of fragmentation that still come down from infinity. In light of this, we hypothesise that, in order for additionally interesting phenomena to emerge, we must choose a more extreme form of fragmentation to `compete' against Kingman's coalescent.  

The `fast' fragmentation-coalescence process that we will work with follows the dynamics of Kingman's coalescent, but with the modification that each block in the system, at constant rate $\lambda>0$, is shattered into its constituent elements.
A more precise definition and construction
of the process  as an exchangeable fragmentation-coalescent process  will be given in Section 2, but for now, we are interested in a Markov chain $N := (N(t): t\geq 0)$  on $\mathbb{N}\cup\{\infty\}$, which represents the number of blocks in the fast fragmentation-coalescence process. Its transitions are specified by the $Q$-matrix having entries given by the aforesaid Kingman and fragmentation dynamic,  so that 
\[
Q_{i,j}= \left\{
\begin{array}{ll}
 c\binom{i}{2} &\text{if } j = i-1,\\
 \lambda i & \text{if } j = \infty.
\end{array}
\right.
\]
We are interested in understanding whether the state $\{\infty\}$ is absorbing or recurrent for $N$. That is to say, we want to know whether it is possible to construct a recurrent extension of the process $N$ beyond its first hitting time of $\{\infty\}$, when issued from a point in $\mathbb{N}$. The idea that $N$ `comes down from infinity' is then clearly captured in the notion that the process instantaneously visits $\mathbb{N}$ after entering the state $\{\infty\}$.  In Proposition 15 of  \cite{BerestyckiJ2004}, Berestycki gives sufficient conditions for an EFC process to come down from infinity. In doing so, he makes assumptions which specifically exclude our setting. Specifically, his assumptions (L) and (H) are violated by our model as they require that blocks split into a finite number of of sub-blocks at a fragmentation event.  

It turns out to be more convenient, however, to study the reciprocal process $M:=1/N$. 
The case that there is a recurrent extension of $M$ from 0 corresponds to the ability of the fast fragmentation-coalescence process to come down from infinity. Moreover, if $M=0$ is an absorbing state then the fast fragmentation-coalescence process stays infinite. It transpires that $\theta:=2\lambda/c$ is the quantity that governs this behaviour. Our  main result in this respect is as follows.

\begin{theorem}[Phase transition]\label{comedown}$\mbox{ }$
\begin{enumerate}
\item[(i)]  If $0<\theta<1$, then $M: = (M(t): t\geq 0)$ is a recurrent  Feller process on $ \{1/n: n\in\mathbb{N}\}\cup\{0\}$ such that $0$ is instantaneously regular (that is to say $0$ is a not a holding point) and not sticky (that is to say $\int_{0}^\infty\mathbf{1}_{\{M(s)=0\}}{\rm d}s = 0$ almost surely). 
  
\item[(ii)]  If $\theta\geq1$, then $0$ is an absorbing state for $M$.
\end{enumerate}
\end{theorem}

\begin{figure}[ht]\label{fig}
\begin{center}
\includegraphics[width=0.8\textwidth, trim=40 40 40 0]{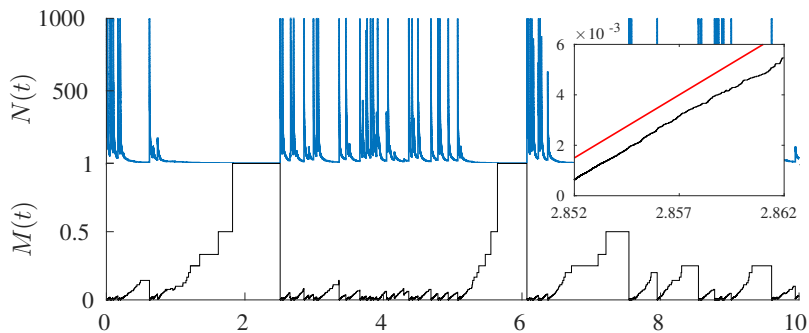}
\caption{A computer simulation of the trajectories of N (blue) and M
(black) restricted to $n = 10^6$ integers, with $c = 1$ and $\lambda= 0.2$. The
inset shows detail of typical behaviour near M = 0; the red line here
illustrates the `speed' predicted by Theorem 
 \ref{CDI}.}
\end{center}
\end{figure}

Theorem \ref{comedown} $(i)$ also alludes to the existence of an excursion theory for the process $M$ away from $0$ (equivalently  $N$ from $\{+\infty\}$). Indeed combined with the recurrence properties of the point $0$ given in Theorem \ref{comedown} (i) standard theory dictates that a local time exists for $M$ at $0$ such that its inverse is a pure jump subordinator with infinite activity. Moreover, accompanying this local time is an excursion measure, $\mathbb{Q}$.  Following classical excursion theory in Chapter XIX.46 of Dellacherie and Meyer \cite{DellacherieMeyer1978}, we can write down an invariant measure for the transition semi-group of $M$ in terms of this excursion measure. On account of the fact that $M$ is a bounded recurrent process in $[0,1]$, we would expect this invariant measure to be a stationary distribution, that is to say, we would expect $M$ to possess an ergodic limit. With some additional work, the next theorem shows that the computations can be pushed even further to obtain an explicit stationary distribution. 

\begin{theorem}[Stationary distribution]\label{localtime}
If $0<\theta<1$, then $M$ has stationary distribution given by 
\[
\rho_M(1/k)=\frac{(1-\theta)}{\Gamma(\theta)}\frac{\Gamma(k-1+\theta)}{\Gamma(k+1)},\quad\text{$k\in\mathbb{N}$},
\]
which is a Beta-Geometric$(1-\theta,\theta)$ distribution and accordingly its probability generating function can be written in the form 
\[
G(s) =  1-(1-s)^{1-\theta} \qquad s\in (0,1).
\]
\end{theorem}

\noindent Another property that can be captured in the recurrent case is that, in the appropriate sense,  the rate of coming down from infinity matches that of Kingman's Coalescent. To this end, let us denote by $\mathbb{P}_{1/n}$, for $n\in\mathbb{N}\cup\{+\infty\}$, the probabilities of $M$.

\begin{theorem}[Speed of coming down from infinity]\label{CDI} Suppose that $0<\theta<1$.
\begin{itemize}
\item[(i)] Let $e^{(\infty)}_{1/k}$ be the expected first hitting time of $1/k$ by $M$ under $\mathbb{P}_0$. Then
\[
e^{(\infty)}_{1/k}=\frac{2}{c(1-\theta)k}.
\]
\item[(ii)]
Let $\epsilon$ denote an excursion of $M$, then
\[
\lim_{t\downarrow0}\frac{t}{\epsilon(t)} = \frac{2}{c}, \quad\text{$\mathbb{Q}$-a.e.}
\]
\end{itemize}
\end{theorem}
We can see the phenomenon in Theorem \ref{CDI} (ii) in Figure \ref{fig} by zooming on the entrance of an excursion of $M$ from $0$. {\color{black} Since the local behaviour of an excursion only has coalescing events as in Kingman's coalescent, {\color{black} albeit that the coalescence rates are slightly adjusted to accommodate for suppressing fragmentation},  and since $M=1/N$ one might expect $(ii)$ in view of  \eqref{speedk}. In fact, the proof of this result will actually use part of the proof in \cite{Aldous1999}. Finally, a natural question in excursion theory is to ask for the Hausdorff dimension of the zero set of $M$.

\begin{theorem}[Hausdorff dimension]\label{Hausthm}
Suppose $0 < \theta < 1$.  Let $Z = \{t: M(t) = 0\}$.  Then the Hausdorff dimension of $Z$ equals $\theta$ almost surely.
\end{theorem}
 }

The rest of the paper is organised as follows. The basis of the analysis as to whether  $\{0\}$ is an absorbing or recurrent state for $M$ centres around the representation of our fast fragmentation-coalescent process as a path functional of a richer  Markov process on the space of exchangeable partitions of the natural numbers.  We therefore commit the next section to addressing this representation. We then prove the four main results in Sections 3, 4,  5 \& 6.

\section{Background}

We start by reviewing the notion of an exchangeable random partition of the natural numbers. 
A partition of $\mathbb{N}$ is a set of subsets $\pi=\{\pi_1,\pi_2,\ldots\}$, such that $\bigcup_{i=1}^\infty \pi_i =\mathbb{N}$ and $\pi_i\cap \pi_j=\emptyset$, $\forall i\neq j$, where, for convenience,  the blocks are ordered by least element. We denote the set of partitions of $\mathbb{N}$ by $\mathcal{P}$. We are also interested in such partitions restricted to $\{1,\ldots,n\}$, which we denote by  $\mathcal{P}_n$. 

It is straightforward to see that a partition $\pi\in\mathcal{P}$ admits an equivalence relation $j\stackrel{\pi}{\sim} k \iff j,k\in\pi_i$ for some $i\in\mathbb{N}$. Suppose $\sigma$ is a permutation of $\mathbb{N}$ with finite support, then it can be applied to $\pi$ to define a new partition $\sigma(\pi)$ using this equivalence relation by letting $j\stackrel{\sigma(\pi)}{\sim} k \iff \sigma(j)\stackrel{\pi}{\sim} \sigma(k)$. In other words, apply $\sigma^{-1}$ to the blocks of $\pi$ to make the blocks of $\sigma(\pi)$, then reorder these blocks by least element. This can be used to define an exchangeable measure.

\begin{definition}
A sigma-finite measure, say $\mu$, on $\mathcal{P}$ is said to be exchangeable  if for all permutations $\sigma$ with finite support, $\sigma(\pi)$ has the same law under $\mu$ as $\pi$.
\end{definition}

We will now state several previous important results to do with exchangeable probability measures. First, we can link measures on the partition space to mass partition measures. A mass paritition is a sequence, in decreasing order, $\textbf{s}=(s_1,s_2,\ldots,)$, with positive elements that sum to at most 1, and the space of all mass partitions is denoted $\mathcal{S}^{\downarrow}$. If we have a mass partition \textbf{s}, we can use this to partition the unit interval $[0,1]$ into subintervals of length $s_1,s_2,\ldots$ We then sample from the uniform distribution on [0,1] recursively $U_1,U_2,\ldots$ to see where they are positioned in this partition of the unit interval. We can define a partition $\pi\in\mathcal{P}$ (known as a paintbox) as follows:
\begin{align*}
j\stackrel{\pi}{\sim}k &\iff \text{$j=k$ or $U_j$ and $U_k$ are in the same block}\\
&\qquad\qquad\quad\:\:\text{in the unit interval partition.}
\end{align*}
Note, it may be the case that $\sum_{i=1}^\infty s_i<1$ and so part of the unit interval is not covered by the partition. In this case, if $U_j$ does not land in any of the $\{s_i: i\geq 1\}$, then $\{j\}$ is taken to be a singleton block of $\pi$. Suppose that for a given $\textbf{s}\in \mathcal{S}^\downarrow$  we write $\rho_{\textbf{s}}$ for the law of the associated paintbox. Kingman \cite{Kingman1978} shows that $\mu$ is an exchangeable probability measure if and only if, for $\pi\in\mathcal{P}$, there exists a probability measure, $\nu$, on $\mathcal{S}^\downarrow$ such that
\begin{equation}
\mu(d\pi) = \int_{\mathcal{S}^\downarrow}\rho_{\bf s}(d\pi)\nu({\rm d}{\bf s}).
\label{equivalence}
\end{equation}
It is also worthy of note that if $\nu$ is not a probability measure, but instead a sigma-finite measure, then $\mu$ is also an exchangeable sigma-finite measure.

As is suggested by  (\ref{equivalence}), 
the asymptotic frequencies of block partitions, defined as 
\begin{equation*}
\abs{\pi_i}=\lim_{n\to\infty}\frac{1}{n}\#\{\pi_i \cap \{1,\ldots,n\}\},
\end{equation*}
exist $\mu$-almost everywhere and, when ranked in decreasing order, correspond in law to $\{s_i: i\geq 1\}$ under $\nu$. A more subtle version of this statement is that all partition blocks contain a either positive fraction of the integers or just a singleton. Moreover  if there are singletons, there are infinitely many of them and the union of all of them has positive asymptotic frequency almost surely. 


\bigskip

Exchangeable fragmentation-coalescence processes (EFC processes) were introduced by Berestycki \cite{BerestyckiJ2004} with applications in physical chemistry and mathematical genetics (amongst others). In short, they superimpose the stochastic mechanisms that drive  homogeneous fragmentation processes (introduced by Bertoin \cite{Bertoin2001}, {\color{black}with other examples from Aldous and Pitman \cite{AldousPitman1998}, Bertoin \cite{Bertoin2000} and Pitman \cite{Pitman1999}}) and exchangeable coalescence processes (first studied by Kingman \cite{Kingman1982} and later generalised by Pitman \cite{Pitman1999}, Sagitov \cite{Sagitov}, Schweinsberg \cite{Schweinsberg2000}, and M\"{o}hle and Sagitov \cite{MS}).

 Little has been done on EFCs in the probabilistic literature beyond the seminal work of Berestycki \cite{BerestyckiJ2004} and we know only of the works \cite{Bertoin2007}
and
\cite{Clement2016}. The former computes the stationary distribution for a specific class of EFCs  {\color{black}(different from ours) }and the latter considers how other some EFCs, which may be regarded as logistic branching processes,  {\color{black}can be analysed} in the light of a duality relation with certain SDEs. The work we present here adds a third contribution in this respect.

 The following definitions are mostly taken from  Berestycki \cite{BerestyckiJ2004}. First, we say that a family of $\mathcal{P}_n$-valued processes, $(\Pi^{(n)})_{n\in\mathbb{N}}$, is compatible if the restriction of $\Pi^{(n+1)}$ to $\{1,\ldots,n\}$ is almost surely equal to $\Pi^{(n)}$. If we have such a family then almost surely this determines a unique $\mathcal{P}$-valued process $\Pi$ such that the restriction of $\Pi$ to $\{1,\ldots,n\}$ is $\Pi^{(n)}$. 
\begin{definition}
A $\mathcal{P}$-valued process $\Pi: = (\Pi(t),t\geq0)$, is an EFC process if
\begin{enumerate}
\item for each $t\geq 0$, $\Pi(t) = (\Pi_1(t), \Pi_2(t),\cdots)$ is an exchangeable partition 
\item its restrictions, $\Pi^{(n)}$, are c\`{a}dl\`{a}g  Markov chains which can only evolve by fragmentation of one block or by coagulation.
\end{enumerate}
\end{definition}

\noindent Note, in the above definition, by c\`{a}dl\`{a}g, we mean with respect to the metric $d(\pi', \pi'') = 1/ \max\{n\in\mathbb{N}: \pi'\cap[n] = \pi''\cap[n]\}$ for $\pi',\pi''\in\mathcal{P}$.

Berestycki \cite{BerestyckiJ2004} shows that all EFC processes can be decomposed in terms of two independent Poisson point processes, one for coalescence and one for fragmentation. Rather than give a full exposition here, we focus on the specific EFC that will lead to the fast fragmentation-coalescent process that we are interested in. 

In order to do so, we need to define the fragmentation and coalescence operators which will form the dynamics of EFC processes. If we have $\pi,\pi'\in\mathcal{P}$ we define the coagulation of $\pi$ by $\pi'$, written ${\bf Coag}(\pi,\pi')$ as a third partition $\pi''$ where
\begin{equation*}
  \pi''_i=\bigcup_{j\in\pi'_i}\pi_j.
\end{equation*}
\begin{sloppypar}In addition, the fragmentation of the $k^{th}$ block of $\pi$ by $\pi'$, written ${\bf Frag}(\pi, \pi', k)$ is a third partition $\tilde{\pi}$ whose blocks are $\pi_i$ for $i\neq k$ and $\pi_k\cap\pi'_j$ for $j\in\mathbb{N}$. These blocks may need to be reordered to ensure that they are still ordered by least element. With these dynamics in place we can define an EFC process.\end{sloppypar}

Next we need to introduce two measures, $C$ and $F$, which give the rates at which fragmentation and coalescence will occur in the aforesaid Poisson point processes. For all $i,j\in\mathbb{N}$ such that $i\neq j$, let $\epsilon_{i,j}$ be the partition which has only one block that is not a singleton, namely the block $\{i,j\}$. We take
$\kappa = \sum_{i<j}\delta_{\epsilon_{i,j}}$ and $C=c\kappa$. This is the exchangeable measure which corresponds to all pairs of blocks coalescing at constant rate $c$. 

For the fragmentation measure,  $F$, as alluded to in the Introduction,  we will work with an extreme case. Specifically, if we define 
$  \textbf{0}=(\{1\},\{2\},\ldots)$, then we take $F =\lambda \delta_{\bf 0}$, for $\lambda>0$. That is to say, each block is fragmented entirely into its constituent singletons. This is a valid exchangeable measure as we can take $\nu_{\textrm{Disl}}=\lambda\delta_{\underline{0}}$, where $\underline{0}$ is the mass partition made of an infinite sequence of zeros. Then, with no erosion present, the paintbox of $\nu_{\textrm{Disl}}$ gives $F$ as required.

Finally we introduce the two independent Poisson point processes and show how they are composed to generate the paths of the exchangeable fast fragmentation-coalescence process (EFFC processes).   Let $PPP_C$ be a Poisson point process on $[0,\infty)\times \mathcal{P}$ with intensity ${\rm d}t\times C(d\pi)$ and let $PPP_F$ be a Poisson point process on $[0,\infty)\times\mathcal{P}\times\mathbb{N}$ with intensity ${\rm d}t\times F(d\pi)\times \#(dk) $ (here $\#$ is the counting measure on $\mathbb{N}$ and we understand the set $[0,\infty)$ to take the role of time). Pick $\Pi(0) = \pi\in\mathcal{P}$ for the starting value of $\Pi$ and we evolve the process $\Pi$ according to arrivals as points in $PPP_C$ and $PPP_F$ as follows. If $t$ is not an atom time of $PPP_C$ or $PPP_F$ then $\Pi(t) = \Pi(t-)$. If $t$ is an atom time of $PPP_C$, then 
$
\Pi_i(t) = {\bf Coag}(\Pi_i(t-),\pi(t)),
$
where $\pi(t)$ is the accompanying mark in $\mathcal{P}$ at the atom time $t$. Finally, if $t$ is a an atom time of $PPP_F$ then 
$
\Pi(t) = {\bf Frag}(\Pi(t-), \pi(t), k(t)),
$
where $(\pi(t),k(t))$ is the accompanying mark in $\mathcal{P}\times\mathbb{N}$ at the atom time $t$. 
A direct consequence of this construction  is that $\Pi$ is a Strong Markov process. We shall denote its probabilities by ${\rm P}_{\pi}$, $\pi\in\mathcal{P}$. Let $[n] = \{1, \cdots, n\}$, for each $n\in\mathbb{N}$. 

One of the main results in Berestycki \cite{BerestyckiJ2004} concerns the existence of a stationary distribution of all EFC processes. In particular, for the EFFC process, we have the following result. 

\begin{theorem}[Berestycki \cite{BerestyckiJ2004}]\label{stdist} For all $\pi\in\mathcal{P}$ such that $\Pi(0) = \pi$,
there exists an exchangeable probability measure $\rho$ on $\mathcal{P}$ such that $\Pi(t)\Rightarrow \rho$ as $t\to\infty$.
\end{theorem}

Recall that we write $N(t)$ for the number of blocks in the system $\Pi(t)$ at time $t\geq 0$ (with the understanding that its value may be $+\infty$). In completing this section, let us remark on how the property of right-continuity of $\Pi$ transfers to the setting of the number of blocks.

\begin{proposition}\label{rtcont}
The process $(N(t), t \geq 0)$ has a right-continuous version.
\end{proposition}

\begin{proof}
Berestycki \cite{BerestyckiJ2004} showed that $\Pi$ is a Feller process and therefore has a right-continuous version.  Proposition \ref{rtcont} does not follow immediately from Berestycki's result because the function which maps a partition in ${\cal P}$ to its number of blocks is not a continuous function with respect to the metric $d(\cdot, \cdot)$ defined above.  However, note that, for all $k \in \mathbb{N}$, the set of times during which the process has $k$ blocks is a union of intervals of the form $[s, t)$.  Therefore, if $N(t) < \infty$, then $N$ is right-continuous at $t$.  That $N$ is right-continuous at $t$ when $N(t) = \infty$ follows from the fact that $\Pi$ has a right-continuous version, and if $(\pi_n)_{n=1}^{\infty}$ is a sequence in ${\cal P}$ that converges to a partition $\pi$ with infinitely many blocks, then the number of blocks of $\pi_n$ tends to infinity as $n \rightarrow \infty$. 
\end{proof}


\section{Proof of Theorem \ref{comedown}}

Let $T = \inf\{t > 0: N(t) < \infty\}$.  The next lemma shows that when $\theta \in(0,1)$, we have $T = 0$ a.s., which establishes that the process instantaneously comes down from infinity. 

\begin{lemma}\label{cdilem}
If $\theta \in(0,1)$ and $\pi \in {\cal P}$, then ${\rm P}_{\pi}(T = 0) = 1$.
\end{lemma}


\begin{proof}
We first look at the probability of hitting the state with 1 block without any fragmentation event occurring when the process starts with $n$ blocks.  Suppose we label this probability $p_{n,1}$.  Recall that $\theta:=2\lambda/c$.  From the definition of the model, it is straightforward to see that
\begin{align}
  p_{n,1}
  =\prod_{k=2}^n\frac{c\binom{k}{2}}{\lambda k +c\binom{k}{2}}
=\prod_{k=2}^n\frac{k-1}{k-1+\theta}
=\frac{\Gamma(n)\Gamma(1+\theta)}{\Gamma(n+\theta)}\label{pn}
\sim\Gamma(1+\theta) n^{-\theta},
\end{align}
 as $n$ tends to infinity. We now find an expression for the expected time until the first fragmentation event. 
 Write $\mathbb{P}_n$ to denote probabilities when the process starts with $n$ blocks.
 First we find the probability that the first fragmentation event occurs when the process is in a state with $k$ blocks, which we denote $r^{(n)}_k$. To this end, define $\tau_k = \inf\{t>0 : N(t) =k\}$  for $k\in\mathbb{N}\cup\{+\infty\}$, and see that
\begin{eqnarray}
r_{k}^{(n)}&=  &\mathbb{P}_n(N(\tau_\infty-) = k)=\frac{\lambda k}{\lambda k +c\binom{k}{2}}\prod_{j=k+1}^n\frac{c\binom{j}{2}}{\lambda j +c\binom{j}{2}}\notag\\
&=&\frac{\theta}{k-1+\theta}\prod_{j=k+1}^n\frac{j-1}{j-1+\theta}=\frac{\theta}{k-1+\theta}\frac{\Gamma(k+\theta)\Gamma(n)}{\Gamma(n+\theta)\Gamma(k)}\notag\\
&=&\frac{\theta\Gamma(n)\Gamma(k-1+\theta)}{\Gamma(n+\theta)\Gamma(k)}.
\label{rnk}
\end{eqnarray}
Next we aim to bound the expected time for the process to reach a state with $k$ blocks if there is no fragmentation event. Note that $N$ has skip-free downward paths so that 
\begin{eqnarray}
t_{k}^{(n)}&=&  \mathbb{E}_n[\tau_k |\tau_k<\tau_\infty]=\sum_{j=k+1}^n\mathbb{E}_j[\tau_{j-1}| \tau_{j-1}<\tau_\infty]\notag\\
&=&\sum_{j=k+1}^n\frac{2}{2\lambda j + c j(j-1)}\leq\sum_{j=k+1}^n\frac{2}{c j(j-1)}\notag\\
&\leq&\frac{2}{c k}.\label{tnk}
\end{eqnarray}
We can combine these last two quantities to bound the expected time until the first fragmentation event. We split over the number of blocks when the first fragmentation event can occur to see that, with the help of \eqref{rnk} and \eqref{tnk}, 
\begin{eqnarray*}
  \mathbb{E}_n[\tau_\infty]&=&\sum_{k=1}^n\left(t_{k}^{(n)}+\frac{1}{\lambda k +c\binom{k}{2}}\right) r_{k}^{(n)}\\
&\leq&\sum_{k=1}^n\left(\frac{2}{ck}+\frac{1}{\lambda k +c\binom{k}{2}}\right)\frac{\theta\Gamma(n)\Gamma(k-1+\theta)}{\Gamma(n+\theta)\Gamma(k)}\\
&\leq&\frac{\Gamma(n)}{\lambda c\Gamma(n+\theta)}\sum_{k=1}^n\frac{(2\lambda+c)\theta\Gamma(k-1+\theta)}{\Gamma(k+1)}.
\end{eqnarray*}

For the next part of the proof, consider a slightly different Markov process on $\mathcal{P}_n$, say $\widehat{\Pi}^{(n)}$, where the rates of coalescence and fragmentation are the same as $\Pi^{(n)}$ (the process $\Pi$ restricted to $\{1, \dots, n\}$), but when fragmentation occurs, we return to the state $(\{1\}, \{2\}, \cdots, \{n\})$.
We can consider, for $\widehat{\Pi}^{(n)}$, the number of times the system attempts to descend to the state $[n]$ in which all integers are in a single block from the initial state $(\{1\}, \{2\}, \cdots,\{n\})$. (The `failure' event corresponds to the event  that a fragmentation occurs during a sojourn from singletons  to $[n]$.) It is   a geometric random variable with success rate $p_{n,1}$ and the expected value of this random variable (expected number of attempts until success) will be $p_{n,1}^{-1}$.   The above tells us the expected amount of time each failure will take, so recalling that ${\bf 0}$ denotes the partition of $\mathbb{N}$ into singletons, we have $${\rm E}_{{\bf 0}}[\text{Time for $\widehat{\Pi}^{(n)}$ to hit $[n]$}] \leq p_{n,1}^{-1} \mathbb{E}_n[\tau_\infty].$$  The expected time for $\widehat{\Pi}^{(n)}$ to hit $[n]$ could only decrease if the process started from a different initial state.  Furthermore, it is easy to see that the expected time to hit $[n]$ from any given initial state is larger for $\widehat{\Pi}^{(n)}$ than for $\Pi^{(n)}$.
Therefore, letting $\pi \in {\cal P}$ be a partition with infinitely many blocks, we have
$$e_{1}^{(n)}:= {\rm E}_{\pi}[\text{Time  for $\Pi^{(n)}$ to hit $[n]$}]\\
\leq p_{n,1}^{-1}\mathbb{E}_n[\tau_\infty].$$
Hence, by \eqref{pn}, we have
\begin{eqnarray*}
 e_{1}^{(n)}&\leq&\frac{1}{\lambda c}\left(\frac{\Gamma(n+\theta)}{\Gamma(n)\Gamma(1+\theta)} \right)\frac{\Gamma(n)}{\Gamma(n+\theta)}\sum_{k=1}^n\frac{(2\lambda+c)\theta\Gamma(k-1+\theta)}{\Gamma(k+1)}\\
&\leq& \frac{1}{\lambda c\Gamma(\theta)} \sum_{k=1}^{n}\frac{(2\lambda+c)\Gamma(k-1+\theta)}{\Gamma(k+1)}\\
&\leq& D\sum_{k=1}^{n}k^{\theta-2}\\
&\leq& D\sum_{k=1}^{\infty}k^{\theta-2},
\end{eqnarray*}
where $D$ is a constant that does not depend on $n$ or on the initial state $\pi$. The above is finite if and only if $0<\theta<1$, uniformly for all $n$. It follows that if $\theta\in(0,1)$, then $\sup_n e^{(n)}_1<\infty$. 

Now let $\tau_1^{(n)} = \inf\{t: \Pi^{(n)}(t) = [n]\}$ be the first time that the integers $1, \dots, n$ are all in the same block.  Then $\tau^{(1)}_1\leq\tau^{(2)}_1\leq\cdots,$ 
and so $\tau^{(n)}_1$ increases to some limit $\tau^{(\infty)}_1$. As $\sup_{n}e^{(n)}_1<\infty$, it follows from the Monotone Convergence Theorem that ${\rm E}_{\pi}[\tau^{(\infty)}_1]<\infty$ and so $\tau^{(\infty)}_1<\infty$ a.s.  Hence, $N(\tau^{(\infty)}_1)=1$, and so $\tau^{(\infty)}_1=\tau_1$.  It follows that $\tau_1$ and thus $T$ are almost surely finite.

To prove that $T=0$ almost surely,  we bound the expected time for $N$ to hit $k$ from initial state $n$ uniformly in $n$ and prove that this converges to 0 as $k\to\infty$. 
The proof is very similar to the argument above. To this end, let $p_{n,k}$ be the probability of $N$ hitting the state with $k$ blocks before a fragmentation event occurs. A similar calculation to \eqref{pn} gives 
\begin{equation}
  p_{n,k}=\frac{\Gamma(k+\theta)\Gamma(n)}{\Gamma(n+\theta)\Gamma(k)}.\label{pnk}
\end{equation}
Then using the same argument as for $ e_{1}^{(n)}$, the expected time to hit a state with $k$ blocks, written $ e_{k}^{(n)}$, satisfies
\begin{eqnarray*}
   e_{k}^{(n)}&\leq& p_{n,k}^{-1}\mathbb{E}_n[\tau_\infty]\notag\\
&\leq& \left(\frac{\Gamma(n+\theta)\Gamma(k)}{\Gamma(k+\theta)\Gamma(n)}\right)\frac{\Gamma(n)}{\lambda c\Gamma(n+\theta)}\sum_{j=1}^n\frac{(2\lambda +c)\theta\Gamma(j-1+\theta)}{\Gamma(j+1)}\notag\\
&\leq&\frac{\Gamma(k)}{\lambda c\Gamma(k+\theta)}\sum_{j=1}^{\infty}\frac{(2\lambda+c)\theta\Gamma(j-1+\theta)}{\Gamma(j+1)}, \label{zero}
\end{eqnarray*}
and so
\begin{equation*}
\sup_n e^{(n)}_k\leq\frac{\Gamma(k)}{\lambda c\Gamma(k+\theta)}\sum_{j=1}^{\infty}\frac{(2\lambda+c)\theta\Gamma(j-1+\theta)}{\Gamma(j+1)}\to0,
\end{equation*}
as $k\to\infty$, because the series is finite for $0<\theta<1$.  Let $\tau_k^{(n)}$ denote the first time the process $\Pi_n$ has $k$ blocks.  Then $\tau^{(k)}_k\leq\tau^{(k+1)}_k\leq\cdots$, so $\tau^{(n)}_k$ increases to some limit $\tau^{(\infty)}_k$. Also, by a similar argument to one given  earlier, $N(\tau^{(\infty)}_k)=k$ and so $\tau^{(\infty)}_k=\tau_k$.  Thus,
\begin{equation*}
0 \leq {\rm E}_{\pi}[T] \leq e^{(\infty)}_k\to0,
\end{equation*}
as $k\to\infty$.  It follows that $T = 0$ almost surely, as required.
\end{proof}

The next result shows that, when the process starts from the partition of the positive integers into singletons, although the process immediately comes down from infinity, there are also fragmentation events at arbitrarily small times which cause the number of blocks to become infinite.


\begin{lemma}\label{regularlem}
Let $S = \inf\{t>0:N(t)=\infty\}$.  If $\theta\in(0,1)$ and $\pi$ is a partition with infinitely many blocks, then ${\rm P}_{\pi}(S=0)=1$.
\end{lemma}

\begin{proof}
Let $F_k$ be the number of fragmentations that occur before the first time the number of blocks reaches $k$,
so that $p_{n,k} = \mathbb{P}_n(F_{k}=0)$ is the probability that $N$ drops from $n$ to $k$ without a fragmentation occurring. When a fragmentation occurs, the process $N$ must first return to a state with $n$ blocks before it can reach a state with $k$ blocks. {\color{black}Hence, we can conclude that, for all $n \geq k$, the random variable $F_{k}$ stochastically dominates a geometric random variable with success probability $p_{n,k}$. Therefore, appealing to \eqref{pnk}, for all $j, k \in \mathbb{N}$, and $n\geq k$ we have 
\begin{equation*}
{\rm P}_{\pi}(F_{k}>j) \geq \left(1-\frac{\Gamma(k+\theta)\Gamma(n)}{\Gamma(n+\theta)\Gamma(k)}\right)^j.
\end{equation*}
Hence, as the right-hand side converges to 1 as $n\to\infty$, we can therefore conclude that for all $k \in \mathbb{N}$, the number of fragmentation events that occur before the process reaches any state with $k$ blocks is infinite almost surely.}  The result follows.
\end{proof}


\begin{lemma}\label{stickylem}
If $\theta\in(0,1)$ and $\pi \in {\cal P}$, then $\int_0^{\infty} {\bf 1}_{\{N(t) = \infty\}} \: {\rm d}t = 0, \: \:{\rm P}_{\pi}$-almost surely.
\end{lemma}

\begin{proof}
Let $N_n(t)$ be the number of blocks in the partition $\Pi^{(n)}(t)$.  For $k, n \in \mathbb{N} \cup \{+ \infty\}$ with $2 \leq k \leq n$, let $$g_k^{(n)} = {\rm E}_{\pi} \bigg[ \int_0^{\tau_1^{(n)}} {\bf 1}_{\{N_n(t) = k\}} \: {\rm d}t \bigg]$$ be the expected amount of time for which the process $\Pi^{(n)}$ has $k$ blocks, before the time $\tau_1^{(n)}$.  To calculate $g_k^{(\infty)}$, note that each time the process $N$ visits the state $k$, it has probability $p_{k,1}$ of reaching the state $1$ before fragmentation, and if there is a fragmentation the process must return to $k$ before reaching the state $1$.  Therefore, assuming there are initially at least $k$ blocks, the process makes $p_{k,1}^{-1}$ visits to $k$ on average before time $\tau_1^{(\infty)}$, and thus $$g_k^{(\infty)} = \frac{1}{p_{k,1}} \cdot \frac{1}{\lambda k + c \binom{k}{2}}.$$  When the process $N_n$ visits $k$, there is still probability $p_{k,1}$ that the process reaches the state with one block before fragmentation, but the process could have fewer than $k$ blocks after fragmentation in which case it could still have another chance to reach the state $1$ before returning to $k$.  It follows that $$g_k^{(n)} \leq g_k^{(\infty)}$$ for $2 \leq k \leq n$.  Thus, $${\rm E}_{\pi}[\tau_1^{(\infty)}] = \sum_{k=2}^{\infty} g_k^{(\infty)} + g_{\infty}^{(\infty)} \geq \sum_{k=2}^n g_k^{({n})} + g_{\infty}^{(\infty)} = {\rm E}_{\pi}[\tau_1^{(n)}] + g_{\infty}^{(\infty)}.$$  Because the times $\tau_1^{(n)}$ increase to $\tau_1^{(\infty)}$, which has finite mean as shown in the proof of Lemma \ref{cdilem}, it follows by letting $n \rightarrow \infty$ and using the Monotone Convergence Theorem that $g_{\infty}^{(\infty)} = 0$.  That is, with probability one, the set of times that $N$ spends in the state $\infty$ before time $\tau_1^{(\infty)}$ has Lebesgue measure zero.  This is sufficient to establish the result.
\end{proof}

To prove part (i) of Theorem \ref{comedown}, it remains to show that $M$ and $N$ are strong Markov processes.  It is known from results of Berestycki \cite{BerestyckiJ2004} that the partition-valued EFFC process $\Pi$ is a Feller process.  However, while the processes $M$ and $N$ clearly evolve in a Markovian way when there are only finitely many blocks, one could be concerned about whether the Markov property holds when there are infinitely many blocks, especially in view of the unusual behavior described in Lemmas \ref{cdilem} and \ref{regularlem}.  In particular, there is the question of whether knowing that the partition has infinitely many blocks provides sufficient information about the partition to determine how the number of blocks evolves in the future.  The lemma below settles this question.

\begin{lemma}\label{Fellerlem}
If $\theta(0,1)$, then $(M(t), t \geq 0)$ and $(N(t), t \geq 0)$ are Feller processes.
\end{lemma}

\begin{proof}
For $k \in \mathbb{N}$ and $t \geq 0$, let $S_k(t) = \inf\{u: \int_0^u {\bf 1}_{\{N(s) \leq k\}} \: ds > t \}.$  Then let $${\hat N}_k(t) = N(S_k(t)), \qquad t\geq 0.$$  Note that the process ${\hat N}_k$ is the same as the original process $N$, except that the periods during which the partition has more than $k$ blocks are cut out.  After every fragmentation event, the process ${\hat N}_k$ jumps to $k$.  Therefore, $({\hat N}_k(t), t \geq 0)$ is a continuous-time Markov chain with state space $\{1, \dots, k\}$ and transition rates $\hat{Q}_{j, j-1}
= c \binom{j}{2}$ for $2 \leq j \leq k$ and $
\hat{Q}_{j,k}= \lambda j$ for $1 \leq j \leq k-1$.  Let $p^k_t(i, j) = P({\hat N}_k(s+t) = j|{\hat N}_k(s) = i)$,  $t\geq 0$, $i,j
\in\{1,\dots, k\}$,  denote the transition probabilities associated with this chain.

Because $\int_0^{\infty} {\bf 1}_{\{N(t) = \infty\}} \: {\rm d}t = 0$ a.s. by Lemma \ref{stickylem}, it follows that for all $t \geq 0$, we have $S_k(t) \downarrow t$ a.s. as $k \rightarrow \infty$.  Since $(N(t), t \geq 0)$ is right-continuous by Proposition \ref{rtcont}, it follows that ${\hat N}_k(t) \rightarrow N(t)$ a.s. as $k \rightarrow \infty$.  Therefore, for all times $t_1 < \dots < t_m$ and positive integers $j_1, \dots, j_m$, we have
$$\lim_{k \rightarrow \infty} P({\hat N}_k(t_1) = j_1, \dots, {\hat N}_k(t_m) = j_m) = P(N(t_1) = j_1, \dots, N(t_m) = j_m)$$ by the Dominated Convergence Theorem.  By applying this result when $m = 1$ and the initial partition has $i$ blocks, we obtain for all $i \in \mathbb{N}$, $j \in \mathbb{N}$, and $t > 0$, the existence of the limit
$$p_t(i, j) := \lim_{k \rightarrow \infty} p_t^k(i, j).$$  Likewise, by considering an initial condition in which the partition has infinitely many blocks, we obtain for all $j \in \mathbb{N}$ and $t > 0$ the existence of the limit $$p_t(\infty, j) := \lim_{k \rightarrow \infty} p_t^k(k, j).$$  Also, because $\int_0^{\infty} {\bf 1}_{\{N(t) = \infty\}} \: {\rm d}t = 0$ a.s., it is not hard to see that ${\rm P}_{\pi}(N(t) = \infty) = 0$ for all $t > 0$ and $\pi \in {\cal P}$.  Therefore, for all $t > 0$, we let $p_t(i, \infty) := 0$ for all $i \in \mathbb{N} \cup \{+\infty\}$.  It then follows that, if $\pi$ has $i$ blocks, then for $j_1, \dots, j_m \in \mathbb{N} \cup \{+\infty\}$,
\begin{eqnarray*}
\lefteqn{{\rm P}_{\pi}(N(t_1) = j_1, \dots, N(t_m) = j_m)}\\
& &\qquad = p_{t_1}(i, j_i) p_{t_2 - t_1}(j_1, j_2) \dots p_{t_m - t_{m-1}}(j_{m-1}, j_m).
\end{eqnarray*}
Thus, $(N(t), t \geq 0)$ is a continuous-time Markov process with transition probabilities $p_t$.

It remains to check that $N$, and therefore $M$, is Feller.  Let $f:\mathbb{N}\cup\{+\infty\}\rightarrow(0,\infty)$ be a continuous function, which in this setting means that $\lim_{n \rightarrow \infty} f(n) = f(\infty)$.  Note that the function $f$ must be bounded.  Using $\mathbb{P}_n$ to denote the law of $N$ started from $n$, we need to show that
\begin{enumerate}
\item For all $n \in \mathbb{N} \cup \{+\infty\}$, we have $\lim_{t \rightarrow 0} \mathbb{E}_n[f(N(t))] = f(n)$.

\item For all $t > 0$, the function $n \mapsto \mathbb{E}_n[f(N(t))]$ is continuous.
\end{enumerate}
The first of these claims follows immediately from the right continuity of $(N(t), t \geq 0)$, see Proposition \ref{rtcont},
the boundedness of $f$, and the Dominated Convergence Theorem.   To prove the second claim, we need to show that $\lim_{n \rightarrow \infty} \mathbb{E}_n[f(N(t))] = \mathbb{E}_{\infty}[f(N(t))]$.  It suffices to show that for all $j \in \mathbb{N}$ and $t > 0$, we have $$\lim_{n \rightarrow \infty} p_t(n,j) = p_t(\infty, j).$$  For $n \leq k < \infty$, let $\zeta_{n,k} = \inf\{t: {\hat N}_k(t) = n\}$.  Observe that $S_k(t + \zeta_{n,k}) \downarrow t + \tau_n$ as $k \rightarrow \infty$.  Also, if the initial partition has infinitely many blocks, then $\tau_n \downarrow 0$ as $n \rightarrow \infty$ by Lemma \ref{cdilem}.  Therefore, using the right continuity on $(N(t), t \geq 0)$ in the first two lines and the strong Markov property of $({\hat N}_k(t), t \geq 0)$ in the third line, we get
\begin{align*}
p_t(\infty, j) &= \lim_{n \rightarrow \infty} {\rm P}_{{\bf 0}}(N(t + \tau_n) = j) \\
&= \lim_{n \rightarrow \infty} \lim_{k \rightarrow \infty} {\rm P}_{{\bf 0}}({\hat N}_k(t + \zeta_{n,k}) = j) \\
&= \lim_{n \rightarrow \infty} \lim_{k \rightarrow \infty} p_t^k(n,j) \\
&= \lim_{n \rightarrow \infty} p_t(n,j),
\end{align*}
which completes the proof.
\end{proof}

\begin{proof}[Proof of Theorem \ref{comedown}]
The process $M$ is a Feller process, and thus strong Markov, by Lemma \ref{Fellerlem}.  That $0$ is a regular point of $M$ follows from Lemma \ref{regularlem}, and that $0$ is not a holding point follows from Lemma \ref{cdilem}.  That $0$ is non-sticky is a consequence of Lemma \ref{stickylem}.  The proof of part $(i)$ of Theorem \ref{comedown} is now complete.

To prove part $(ii)$ of the theorem, we consider the excursions from infinity of the process $N$ and show that no such excursions can exist. Let $u_k$ be the expected waiting time in a state with $k$ blocks, specifically
\begin{equation*}
  u_k=\frac{2}{2\lambda k + c k(k-1)}.
\end{equation*}
Fix $t>0$, and let $s_k(t)$ be the expected time spent in a state with $k$ blocks during excursions that started before time $t$. Also, let $E_k(t)$ be the expected number of excursions started before time $t$ that reach a state with $k$ blocks. Then 
\begin{equation*}
  s_k(t)=E_k(t)\cdot u_k,
\end{equation*}
and in particular
\begin{align*}
  s_1(t)&=E_1(t)\cdot u_1\\
&=\frac{1}{\lambda}E_k(t)\cdot p_{k,1},
\end{align*}
because once an excursion has reached a state with $k$ blocks, the probability that it reaches a state with just 1 block is $p_{k,1}$ from \eqref{pn}.
Hence, for all $k\in\mathbb{N}$ we have
\begin{equation*}
  \frac{s_k(t)}{s_1(t)}=\frac{\lambda u_k}{p_{k,1}},
\end{equation*}
as long as $s_k(t)>0$ for all $k$. Thus, assuming for a contradiction, that this is the case we have that
\begin{align*}
  \sum_{k=1}^{\infty}\frac{s_k(t)}{s_1(t)}&=\frac{2\lambda}{c}\sum_{k=1}^{\infty}\frac{1}{p_{k,1}}\frac{1}{k(k-1+\theta)}\\
&=\theta\sum_{k=1}^{\infty}\frac{\Gamma(k+\theta)}{\Gamma(k)\Gamma(1+\theta) k(k-1+\theta)},
\end{align*}
which is $+\infty$ if, and only if, $\lambda/c\geq1/2$, as then $\theta\geq1$. However, the expected total time spent on excursions that start before time $t$ is finite almost surely, so we must conclude that $s_k(t)=0$ for all $k\in\mathbb{N}$. That is $\Pi$ stays infinite almost surely.
\end{proof}
\newpage
\section{Proof of Theorem \ref{localtime}}

\begin{proof}[Proof of Theorem \ref{localtime}]
 On account of the fact that $M$ is a Feller (and hence strong Markov) process, standard theory now allows us to invoke the existence of a local time at zero for $M$, denoted by $L = (L_t: t\geq 0)$, such that the right inverse of $L$, say $L^{-1}$, is a pure jump subordinator.  As $0$ is instantaneous regular for $M$, it follows that $L^{-1}$ has   infinite activity. The periods of time where the process is in a state with finitely many blocks correspond to the the excursions away from zero for $M$. Moreover, the fact that the state $0$ is not sticky for $M$ implies that $L^{-1}$ is pure jump with no drift component. Classical results from excursion theory, cf. Chapter XIX.46 of Dellacherie and Meyer \cite{DellacherieMeyer1978}, give us a way to construct the stationary distribution, $\rho_M$, of $M$ using the excursion measure. To this end, let us introduce the canonical space of excursions, that is c\`adl\`ag measurable mappings  $\epsilon: (0,\zeta]\rightarrow \{1/n: n\in\mathbb{N}\}\cup\{0\}$, where $\zeta = \inf\{t>0: \epsilon(t) = 0\}$ and $\lim_{t\downarrow 0}\epsilon(t) = 0$, with associated excursion measure (not a probability measure) $\mathbb{Q}$.

First note, that $\rho_M$ has no atom at zero as it is not a sticky point, equivalently, the inverse local time has no linear component. Again, referring to Chapter XIX.46 of Dellacherie and Meyer \cite{DellacherieMeyer1978}, we have that,  when $\mathbb{Q}(\zeta)<\infty$, for $k\in\mathbb{N}$,

\begin{equation*}
\rho_M(1/k)=\frac{\mathbb{Q}\left(\int_0^{\zeta}\mathbf{1}_{\{\epsilon(t)=1/k\}}{\rm d}t\right)}{\mathbb{Q}(\zeta)}.
\end{equation*}
 It's difficult to compute this quantity directly starting from zero, however, we can appeal to a technique that uses the strong Markov property for excursions; see Section VI.48 of \cite{RW}. Let $\sigma_{1/n}=\inf\{t>0:\epsilon(t)=1/n\}$, then we see that, for $n>k$,
\begin{equation*}
\mathbb{Q}\left(\int_0^{\zeta}\mathbf{1}_{\{\epsilon(t)=1/k\}}{\rm d}t\right)=\mathbb{Q}(\sigma_{1/n}<\zeta)\mathbb{E}_{n}\left[\int_0^{S}\mathbf{1}_{\{N(t)=k\}}{\rm d}t\right],
\end{equation*}
where $\mathbb{E}_{n}$ is expectation of the process given we start in any state with $n$ blocks and $S$ is the time of the first fragmentation event and therefore matches $\zeta$. In any excursion, $\epsilon$, $\Pi$ visits a state with $k$ blocks only once at most. This is because once we are in a state with fewer than $k$ blocks there must be a fragmentation event before $\Pi$ can be in a state with $k$ blocks again, which means it will be a different excursion. Hence, the above expectation can break down into the probability of reaching a state with $k$ blocks given you start in one with $n$, times the expected amount of time spent in a state $k$ blocks before leaving, thus, appealing to \eqref{pnk} we see that
\begin{align}
\mathbb{Q}\left(\int_0^{\zeta}\mathbf{1}_{\{\epsilon(t)=1/k\}}{\rm d}t\right)&=\mathbb{Q}(\sigma_{1/n}<\zeta)p_{n,k}\frac{2}{ck(k-1)+2\lambda k}\notag\\
&=\mathbb{Q}(\sigma_{1/n}<\zeta)\frac{\Gamma(n)}{\Gamma(n+\theta)}\frac{2}{c}\frac{\Gamma(k-1+\theta)}{\Gamma(k+1)}.\label{rhoM}
\end{align}
Therefore, as the left-hand side is positive and finite, we may conclude that there exists a constant $C$ such that
\begin{equation}\label{qsigma}
\mathbb{Q}(\sigma_{1/n}<\zeta)\frac{\Gamma(n)}{\Gamma(n+\theta)}=C.
\end{equation}
As excursion measures are only defined up to a multiplicative constant we can take $C=1$ without loss of generality. In addition, from \eqref{rhoM}, we have that 
\begin{equation*}
\rho_M(1/k)\propto \frac{2}{c}\frac{\Gamma(k-1+\theta)}{\Gamma(k+1)}.
\end{equation*}
Note that the right hand side of this equation is $O(k^{-(2-\theta)})$ and so, as $\theta\in(0,1)$, we can normalise this into a probability measure using the constant $Z^{-1}$ where
\[
Z^{-1}:=\frac{2}{c}\sum_{k=1}^\infty\frac{\Gamma(k-1+\theta)}{\Gamma(k+1)}=\frac{\Gamma(1+\theta)}{\lambda(1-\theta)},
\]
which gives the desired result.
\end{proof}

\section{Proof of Theorem \ref{CDI}}
\begin{proof}[Proof of Theorem \ref{CDI}]
$(i)$ Basic Markov chain theory tells us that the stationary distribution probabilities for each state are the inverse of the mean return time from that state {\color{black}multiplied by the jump rate from that state.} Hence  
\begin{align*}
\frac{1}{\rho_M(1/k)}&=\left(c\binom{k}{2}+\lambda k\right)\mathbb{E}_{1/k}[\text{Time for $M$ to return to $1/k$}] \\
&=\left(c\binom{k}{2}+\lambda k\right)\left(\mathbb{E}_{1/k}[\text{Time to hit 0}]+e^{(\infty)}_{1/k}\right),
\end{align*}
as to return to a state with $k$ blocks you must first fragment, then come down from infinity and reach a state with $k$ blocks once more. Rearranging this shows that
\begin{equation*}
e^{(\infty)}_{1/k}=\frac{1}{\rho_M(1/k)(c\binom{k}{2}+\lambda k)}-\mathbb{E}_{1/k}[\text{Time to first fragmentation event}].
\end{equation*}
Using Theorem \ref{localtime}, we see that
\begin{equation*}
e^{(\infty)}_{1/k}=\frac{\Gamma(\theta)\Gamma(k+1)}{(1-\theta)\Gamma(k-1+\theta)(c\binom{k}{2}+\lambda k)}-\sum_{j=1}^k p_{k,j}\frac{\lambda j}{c\binom{j}{2}+\lambda j}\sum_{i=j}^k\frac{1}{c\binom{i}{2}+\lambda i},
\end{equation*}
as the expected time to the first fragmentation event can be split over how many blocks you have just before you fragment. Therefore, using \eqref{pnk},
\begin{align*}
e^{(\infty)}_{1/k}&=\frac{2\Gamma(\theta)\Gamma(k)}{c(1-\theta)\Gamma(k-1+\theta)(k-1+\theta)}\\&\qquad\qquad-\sum_{j=1}^k \frac{\Gamma(j+\theta)\Gamma(k)}{\Gamma(k+\theta)\Gamma(j)}\frac{\theta}{j-1+\theta}\sum_{i=j}^k\frac{2}{ci(i-1)+2\lambda i}\\
&=\frac{2\Gamma(\theta)\Gamma(k)}{c(1-\theta)\Gamma(k+\theta)}-\theta\sum_{j=1}^k \frac{\Gamma(j-1+\theta)\Gamma(k)}{\Gamma(k+\theta)\Gamma(j)}\sum_{i=j}^k\frac{2}{ci(i-1)+2\lambda i}\\
&=\frac{2\Gamma(\theta)\Gamma(k)}{c(1-\theta)\Gamma(k+\theta)}-\frac{2\theta\Gamma(k)}{c\Gamma(k+\theta)}\sum_{i=1}^k\frac{1}{i(i-1+\theta)}\sum_{j=1}^i\frac{\Gamma(j-1+\theta)}{\Gamma(j)},\\
\intertext{which, by standard results for sums of Gamma functions, gives us that} 
e^{(\infty)}_{1/k}&=\frac{2\Gamma(\theta)\Gamma(k)}{c(1-\theta)\Gamma(k+\theta)}-\frac{2\Gamma(k)}{c\Gamma(k+\theta)}\frac{\Gamma(\theta)\Gamma(k+2)-(k+1)\Gamma(k+\theta)}{(1-\theta)\Gamma(k+2)}\\
&=\frac{2}{c(1-\theta)k},
\end{align*}
as required.

\bigskip

$(ii)$ Let $k\in\mathbb{N}$. Denote by $\mathbb{Q}_k$ the measure $\mathbb{Q}$ conditioned on $\{\zeta>\sigma_{1/k}\}$, where $\zeta$ is the excursion length. {\color{black}As the reasoning below shows, under $\mathbb{Q}_k$, an excursion from} 0 of $M$ looks like a scaled Kingman coalescent, but with slightly accelerated rates, until it reaches a state with $k$ blocks. Hence, we can use Aldous' construction of Kingman's coalescent \cite{Aldous1999}. Define random variables $\varphi_j$, for $j\geq 1$, as 
\begin{equation*}
\varphi_j=\sum_{i=j+1}^{\infty}\xi_i,
\end{equation*}
where $\xi_i$ are independent exponential, with rate $c\binom{i}{2}+\lambda i$. Under the aforementioned conditioning, $\varphi_j$ is the hitting time of a state with $j$ blocks when $j\geq k$. Let $U_j$ be iid uniform random variables on (0,1), $j\geq1$. Then for all $j$, draw a vertical line of length $\varphi_j$ at point $U_j$ on the unit interval. At time $t$, where $\varphi_j<t<\varphi_{j-1}$, $j\geq k$, look at the subintervals of [0,1] with endpoints $\{0,1,U_1,\ldots,U_{j-1}\}$. The lengths of the subintervals have the same distribution as the asymptotic frequencies of the blocks of $\Pi$ conditional on $\zeta>\sigma_{1/k}$, when $\Pi$ has $j$ blocks. 

Then, under the measure $\mathbb{Q}_k$, for large $j$ we have
\begin{equation*}
\mathbb{E}[\varphi_j]=\sum_{i=j+1}^{\infty}\frac{1}{c\binom{i}{2}+\lambda i}\sim\frac{2}{cj}, \text{ and }  \textrm{Var[$\varphi_j$]}=\sum_{i=j+1}^{\infty}\left(\frac{1}{c\binom{i}{2}+\lambda i}\right)^2\sim\frac{4}{3c^2j^3},
\end{equation*}
which is the exact same asymptotic behaviour as for a Kingman coalescent of rate $c$. Hence we may conclude, using Aldous' method as used for the Kingman coalescent case, that
\begin{equation*}
\lim_{t\downarrow0}\frac{t}{\epsilon(t)}=\lim_{t\downarrow0}tN(t)=\frac{2}{c},
\end{equation*}
$\mathbb{Q}_k$-almost surely. As this is independent of $k$, we may conclude this occurs $\mathbb{Q}$-almost everywhere. 
\end{proof}

{\color{black}
\section{Proof of Theorem 4}
{\color{black}Horowitz \cite{Horowitz1972} gives a standard result for the Hausdorff dimension of the range of the inverse local time as the polynomial growth of Laplace exponent. Specifically, 
\begin{equation*}
\dim(Z)=\liminf_{q\to\infty}\frac{\log \Phi(q)}{\log q},
\end{equation*}
where $\Phi$ is the Laplace exponent of the inverse local time. However, for the present case this computation is neither straightforward nor enlightening, so we instead pursue direct upper and lower bounds. } Hence, we will need two well-known facts about the Hausdorff dimension.  The first is the countable stability property, which is discussed in Section 4.1 of \cite{mp}.  If $E = \cup_{i=1}^{\infty} E_i$, then $$\dim(E) = \sup_{i \geq 1} \: \dim(E_i).$$  The second is the mass distribution principle, which is discussed in section 4.2 of \cite{mp}.  If there is a finite nonzero measure $\mu$ on $E$ and positive constants $C > 0$ and $\delta > 0$ such that $\mu(V) \leq C|V|^{\alpha}$ for all closed subsets $V$ of $E$ with $|V| \leq \delta$, then $\dim(E) \geq \alpha$.

\begin{proof}[Proof of Theorem \ref{Hausthm}]
Recall that $\tau_1 = \inf\{t: M(t) = 1\}$.  Let $Z' = Z \cap [0, \tau_1]$.  By the Markov property of the EFFC process and the countable stability property of Hausdorff dimension, to show that $\dim(Z) \leq \theta$ almost surely, it suffices to show that $\dim(Z') \leq \theta$ almost surely.  

Choose $\alpha \in (\theta, 1)$, and fix a large positive integer $k$. Let $G_k$ be the number of excursions of $M$ that reach $1/k$ before the process $M$ reaches $1$, which has a geometric distribution with parameter $p_{k,1}$. Therefore, by (\ref{pn}), there is a positive constant $C_1$ such that $E[G_k] \leq C_1 k^{\theta}$. 

Set $\gamma_{1,k} = 0$, and for positive integers $n$, let $\beta_{n,k} = \inf\{t > \gamma_{n,k}: M(t) = 1/k\}$ and $\gamma_{n+1,k} = \inf\{t > \beta_{n,k}: M(t) = 0\}$. Then, for $i=1,\ldots, G_k$, let $E_{i,k} = [\gamma_{i,k}, \beta_{i,k}]$, the interval of time after the $(i-1)^{th}$ excursion to reach state $1/k$, needed for an excursion to reach state $1/k$ again. Then $E_{1,k}, \dots, E_{G_k,k}$ is a covering of $Z'$. By part (i) of Theorem \ref{CDI}, we have $$\mathbb{E}[|E_{i,k}|] = \frac{2}{c(1-\theta)k}, \quad \text{for all $i=1,\ldots,G_k$.}$$   By Jensen's Inequality, there is a positive constant $C_2$ such that
$$\mathbb{E} \bigg[ \sum_{i=1}^{G_k} |E_{i,k}|^{\alpha} \bigg] = \mathbb{E}[|E_{1,k}|^{\alpha}] \mathbb{E}[G_k] \leq (\mathbb{E}[|E_{1,k}|])^{\alpha} \mathbb{E}[G_k] \leq C_2 k^{\theta - \alpha},$$ which tends to zero as $k \rightarrow \infty$.  By Markov's Inequality, for all $\varepsilon > 0$, we have $$\lim_{k \rightarrow \infty} \mathbb{P} \bigg( \sum_{i=1}^{G_k} |E_{i,k}|^{\alpha} > \varepsilon \bigg) = 0,$$ which is sufficient to establish that $\dim(Z') \leq \alpha$ for all $\alpha\in(\theta,1)$. Thus $\dim(Z')\leq\theta$ and thus $\dim(Z) \leq \theta$ almost surely.

It remains to establish the lower bound on the Hausdorff dimension.  Recall that $L$ denotes the local time at zero of the process $M$, and $L^{-1}$ denotes the inverse local time.  Let $Z^* = Z \cap [0, L^{-1}_1]$.  {\color{black}By monotonicity} it suffices to show that $\dim(Z^*) \geq \theta$ almost surely.  We will use the mass distribution principle.  

Let $\mu$ be the probability measure on $Z^*$ whose distribution function is given by the inverse local time, so 
\begin{equation*}
\mu((s, t]) = L_t - L_s,\quad\text{for $0 \leq s < t \leq L^{-1}_1$.}
\end{equation*}
Choose $\alpha \in (0, \theta)$, and let $\eta = (\theta/\alpha) - 1 > 0$.  Recall from \eqref{qsigma} (with $C$ taken to be 1) that $\mathbb{Q}(\sigma_{1/n} < \zeta) = \Gamma(n + \theta)/\Gamma(n) \sim n^{\theta}$ as $n \rightarrow \infty$.  Therefore, there exists a positive constant $C_3$ such that for all $r \geq 1$, excursions {\color{black} of $M$ from $0$} which reach $1/r$ or higher occur at rate at least $C_3 r^{\theta}$ on the local time scale.  In particular, if $k$ is a positive integer and $1 \leq m \leq 2^k$, then while the accumulated local time at zero is between $(m-1)2^{-k}$ and $m2^{-k}$, the number of excursions which reach height $2^{-k/\alpha}$ or higher has a Poisson distribution with mean at least $2^{-k} \cdot C_3 (2^{k/\alpha})^{\theta} = C_3 2^{\eta k}$.  

Once the number of blocks gets down to $2^{k/\alpha}$, fragmentations are happening at the rate at most $\lambda2^{k/\alpha}$, and so the probability that the excursion lasts for a time $2^{-k/\alpha}$ or larger is at least $e^{-\lambda}$.  It follows that while the accumulated local time at zero is between $(m-1)2^{-k}$ and $m2^{-k}$, the expected number of excursions that last for a time at least $2^{-k/\alpha}$ is at least $C_4 2^{\eta k}$, where $C_4 = C_3 e^{-\lambda}$.  Since we will have $L^{-1}_{m 2^{-k}} - L^{-1}_{(m-1)2^{-k}} < 2^{-k/\alpha}$ only if there are no such excursions, it follows that $$\mathbb{P}\big( L^{-1}_{m 2^{-k}} - L^{-1}_{(m-1)2^{-k}} < 2^{-k/\alpha} \big) \leq e^{-C_4 2^{\eta k}}$$ and therefore $$\mathbb{P} \big( L^{-1}_{m 2^{-k}} - L^{-1}_{(m-1)2^{-k}} < 2^{-k/\alpha} \mbox{ for some }m \in \{1, \dots, 2^k\} \big) \leq 2^k e^{-C_4 2^{\eta k}}.$$  Since $\sum_{k=1}^{\infty} 2^k e^{-C_4 2^{\eta k}} < \infty$, by the Borel-Cantelli Lemma, there almost surely exists a random positive integer $K$ such that for $k \geq K$, we have $$L^{-1}_{m2^{-k}} - L^{-1}_{(m-1)2^{-k}} \geq 2^{-k/\alpha} \mbox{ for all }k \geq K \mbox{ and }m \in \{1, \dots, 2^k\}.$$  Now let $\delta = 2^{-K/\alpha}$.  Let $V$ be a closed subset of $Z^*$ with $|V| \leq \delta$.  Let $t = \inf\{s: s \in V\}$ and $\gamma = |V|$, so $V \subset [t, t + \gamma]$.  We have $$\mu(V) \leq \mu([t, t + \gamma]) = L_{t + \gamma} - L_t.$$  Let $k$ be the positive integer such that $2^{-(k+1)/\alpha} < \gamma \leq 2^{-k/\alpha}$.  Note that because $\gamma \leq \delta$, we have $k \geq K$.  Choose $m$ such that $(m-1)2^{-k} \leq L_t < m 2^{-k}$.  If $m < 2^k$, then $$L_{(m+1)2^{-k}}^{-1} - L_{m2^{-k}}^{-1} \geq 2^{-k/\alpha} \geq \gamma,$$ so $L_{t + \gamma} \leq (m+1)2^{-k}$.  It follows that $$\mu(V) \leq L_{t+\gamma} - L_t \leq (m+1)2^{-k} - (m-1)2^{-k} = 4 \cdot  2^{-(k+1)} \leq 4 \gamma^{\alpha}.$$  Likewise, if $m = 2^k$, then since $t+\gamma \in V \subset Z^*$, we have $L_{t + \gamma} \leq 1$ and therefore $$\mu(V) \leq L_{t+\gamma} - L_t \leq 1 - (m-1)2^{-k} \leq 2 \gamma^{\alpha}.$$
It now follows from the mass distribution principle that $\dim(Z^*) \geq \alpha$ and thus $\dim(Z^*) \geq \theta$.  The proof of Theorem \ref{Hausthm} is now complete.
\end{proof}
}

\section*{Acknowledgements}
{\color{black}We would like to thank Amaury Lambert who, following a presentation of an initial version of this work  at Imperial College London, suggested that an excursion theory could be developed. Also we thank Bat\i{} \c{S}eng\"ul and Cl\'ement Foucart for helpful discussions regarding the intuition behind the main result. We would also like to thank the referees for posing a question with regards the Hausdorff dimension, leading us to find an additional interesting result. 
Part of this work was carried out when JS visited Bath to give a doctoral course, for which we would like to acknowledge support from EPSRC CDT SAMBa grant EP/L015684/1.}

\bibliographystyle{imsart-nameyear}

\begin{thebibliography}{18}

\bibitem[\protect\citeauthoryear{Aldous}{1999}]{Aldous1999}
\begin{barticle}[author]
\bauthor{\bsnm{Aldous},~\bfnm{David~J.}\binits{D.~J.}}
(\byear{1999}).
\btitle{Deterministic and stochastic models for coalescence (aggregation and
  coagulation): a review of the mean-field theory for probabilists}.
\bjournal{Bernoulli}
\bvolume{5}
\bpages{3--48}.
\bdoi{10.2307/3318611}
\bmrnumber{1673235}
\end{barticle}
\endbibitem

\bibitem[\protect\citeauthoryear{Aldous and Pitman}{1998}]{AldousPitman1998}
\begin{barticle}[author]
\bauthor{\bsnm{Aldous},~\bfnm{David}\binits{D.}} \AND
  \bauthor{\bsnm{Pitman},~\bfnm{Jim}\binits{J.}}
(\byear{1998}).
\btitle{The standard additive coalescent}.
\bjournal{Ann. Probab.}
\bvolume{26}
\bpages{1703--1726}.
\bdoi{10.1214/aop/1022855879}
\bmrnumber{1675063}
\end{barticle}
\endbibitem

\bibitem[\protect\citeauthoryear{Berestycki}{2004}]{BerestyckiJ2004}
\begin{barticle}[author]
\bauthor{\bsnm{Berestycki},~\bfnm{Julien}\binits{J.}}
(\byear{2004}).
\btitle{{E}xchangeable fragmentation-coalescence processes and their
  equilibrium measures}.
\bjournal{Electron. J. Probab.}
\bvolume{9}
\bpages{770--824}.
\bdoi{10.1214/EJP.v9-227}
\bmrnumber{2110018}
\end{barticle}
\endbibitem

\bibitem[\protect\citeauthoryear{Berestycki, Berestycki and
  Limic}{2010}]{BerestyckiJNLimic2010}
\begin{barticle}[author]
\bauthor{\bsnm{Berestycki},~\bfnm{Julien}\binits{J.}},
  \bauthor{\bsnm{Berestycki},~\bfnm{Nathana{\"e}l}\binits{N.}} \AND
  \bauthor{\bsnm{Limic},~\bfnm{Vlada}\binits{V.}}
(\byear{2010}).
\btitle{The {$\Lambda$}-coalescent speed of coming down from infinity}.
\bjournal{Ann. Probab.}
\bvolume{38}
\bpages{207--233}.
\bdoi{10.1214/09-AOP475}
\bmrnumber{2599198}
\end{barticle}
\endbibitem

\bibitem[\protect\citeauthoryear{Bertoin}{2000}]{Bertoin2000}
\begin{barticle}[author]
\bauthor{\bsnm{Bertoin},~\bfnm{Jean}\binits{J.}}
(\byear{2000}).
\btitle{A fragmentation process connected to Brownian motion}.
\bjournal{Probab. Theory Related Fields}
\bvolume{117}
\bpages{289--301}.
\bdoi{10.1007/s004400050008}
\bmrnumber{1771665}
\end{barticle}
\endbibitem

\bibitem[\protect\citeauthoryear{Bertoin}{2001}]{Bertoin2001}
\begin{barticle}[author]
\bauthor{\bsnm{Bertoin},~\bfnm{Jean}\binits{J.}}
(\byear{2001}).
\btitle{{H}omogeneous fragmentation processes}.
\bjournal{Probab. Theory Related Fields}
\bvolume{121}
\bpages{301--318}.
\bdoi{10.1007/s004400100152}
\bmrnumber{1867425}
\end{barticle}
\endbibitem

\bibitem[\protect\citeauthoryear{Bertoin}{2007}]{Bertoin2007}
\begin{barticle}[author]
\bauthor{\bsnm{Bertoin},~\bfnm{Jean}\binits{J.}}
(\byear{2007}).
\btitle{{T}wo-parameter {P}oisson-{D}irichlet measures and reversible
  exchangeable fragmentation-coalescence processes}.
\bjournal{Combin. Probab. Comput.}
\bvolume{17}
\bpages{329--337}.
\bdoi{10.1017/S0963548307008784}
\bmrnumber{2410390}
\end{barticle}
\endbibitem

\bibitem[\protect\citeauthoryear{Dellacherie and
  Meyer}{1987}]{DellacherieMeyer1978}
\begin{bbook}[author]
\bauthor{\bsnm{Dellacherie},~\bfnm{Claude}\binits{C.}} \AND
  \bauthor{\bsnm{Meyer},~\bfnm{Paul-Andr{\'e}}\binits{P.-A.}}
(\byear{1987}).
\btitle{Probabilit\'es et potentiel. {C}hapitres {XII}--{XVI}},
\bedition{second} ed.
\bseries{Publications de l'Institut de Math\'ematiques de l'Universit\'e de
  Strasbourg, XIX}.
\bpublisher{Hermann, Paris}.
\bmrnumber{898005 (88k:60002)}
\end{bbook}
\endbibitem

\bibitem[\protect\citeauthoryear{Foucart}{2016}]{Clement2016}
\begin{barticle}[author]
\bauthor{\bsnm{Foucart},~\bfnm{Cl{\'e}ment}\binits{C.}}
(\byear{2016}).
\btitle{On the coming down from infinity of discrete logistic branching
  processes}.
\text{Available at }\href{http://arxiv.org/abs/1605.07039}{arXiv:1605.07039}
\end{barticle}
\endbibitem

\bibitem[\protect\citeauthoryear{Horowitz}{1972}]{Horowitz1972}
\begin{barticle}[author]
\bauthor{\bsnm{Horowitz},~\bfnm{Joseph}\binits{J.}}
(\byear{1972}).
\btitle{Semilinear {M}arkov processes, subordinators and renewal theory}.
\bjournal{Z. Wahrscheinlichkeitstheorie und Verw. Gebiete}
\bvolume{24}
\bpages{167--193}.
\bmrnumber{0322969}
\end{barticle}
\endbibitem

\bibitem[\protect\citeauthoryear{Kingman}{1978}]{Kingman1978}
\begin{barticle}[author]
\bauthor{\bsnm{Kingman},~\bfnm{J.~F.~C.}\binits{J.~F.~C.}}
(\byear{1978}).
\btitle{The representation of partition structures}.
\bjournal{J. London Math. Soc. (2)}
\bvolume{18}
\bpages{374--380}.
\bdoi{10.1112/jlms/s2-18.2.374}
\bmrnumber{509954}
\end{barticle}
\endbibitem

\bibitem[\protect\citeauthoryear{Kingman}{1982}]{Kingman1982}
\begin{barticle}[author]
\bauthor{\bsnm{Kingman},~\bfnm{J.~F.~C.}\binits{J.~F.~C.}}
(\byear{1982}).
\btitle{The coalescent}.
\bjournal{Stochastic Process. Appl.}
\bvolume{13}
\bpages{235--248}.
\bdoi{10.1016/0304-4149(82)90011-4}
\bmrnumber{671034}
\end{barticle}
\endbibitem

\bibitem[\protect\citeauthoryear{M{\"o}hle and Sagitov}{2001}]{MS}
\begin{barticle}[author]
\bauthor{\bsnm{M{\"o}hle},~\bfnm{Martin}\binits{M.}} \AND
  \bauthor{\bsnm{Sagitov},~\bfnm{Serik}\binits{S.}}
(\byear{2001}).
\btitle{A classification of coalescent processes for haploid exchangeable
  population models}.
\bjournal{Ann. Probab.}
\bvolume{29}
\bpages{1547--1562}.
\bdoi{10.1214/aop/1015345761}
\bmrnumber{1880231 (2003b:60134)}
\end{barticle}
\endbibitem

\bibitem[\protect\citeauthoryear{M{\"o}rters and Peres}{2010}]{mp}
\begin{bbook}[author]
\bauthor{\bsnm{M{\"o}rters},~\bfnm{Peter}\binits{P.}} \AND
  \bauthor{\bsnm{Peres},~\bfnm{Yuval}\binits{Y.}}
(\byear{2010}).
\btitle{Brownian motion}.
\bseries{Cambridge Series in Statistical and Probabilistic Mathematics}.
\bpublisher{Cambridge University Press, Cambridge}
\bnote{With an appendix by Oded Schramm and Wendelin Werner}.
\bdoi{10.1017/CBO9780511750489}
\bmrnumber{2604525}
\end{bbook}
\endbibitem

\bibitem[\protect\citeauthoryear{Pitman}{1999}]{Pitman1999}
\begin{barticle}[author]
\bauthor{\bsnm{Pitman},~\bfnm{Jim}\binits{J.}}
(\byear{1999}).
\btitle{Coalescents with multiple collisions}.
\bjournal{Ann. Probab.}
\bvolume{27}
\bpages{1870--1902}.
\bdoi{10.1214/aop/1022677552}
\bmrnumber{1742892}
\end{barticle}
\endbibitem

\bibitem[\protect\citeauthoryear{Rogers and Williams}{2000}]{RW}
\begin{bbook}[author]
\bauthor{\bsnm{Rogers},~\bfnm{L.~C.~G.}\binits{L.~C.~G.}} \AND
  \bauthor{\bsnm{Williams},~\bfnm{David}\binits{D.}}
(\byear{2000}).
\btitle{Diffusions, {M}arkov processes, and martingales. {V}ol. 2}.
\bseries{Cambridge Mathematical Library}.
\bpublisher{Cambridge University Press, Cambridge}
\bnote{It{\^o} calculus, Reprint of the second (1994) edition}.
\bdoi{10.1017/CBO9781107590120}
\bmrnumber{1780932 (2001g:60189)}
\end{bbook}
\endbibitem

\bibitem[\protect\citeauthoryear{Sagitov}{1999}]{Sagitov}
\begin{barticle}[author]
\bauthor{\bsnm{Sagitov},~\bfnm{Serik}\binits{S.}}
(\byear{1999}).
\btitle{The general coalescent with asynchronous mergers of ancestral lines}.
\bjournal{J. Appl. Probab.}
\bvolume{36}
\bpages{1116--1125}.
\bmrnumber{1742154 (2001f:92019)}
\end{barticle}
\endbibitem

\bibitem[\protect\citeauthoryear{Schweinsberg}{2000}]{Schweinsberg2000}
\begin{barticle}[author]
\bauthor{\bsnm{Schweinsberg},~\bfnm{Jason}\binits{J.}}
(\byear{2000}).
\btitle{Coalescents with simultaneous multiple collisions}.
\bjournal{Electron. J. Probab.}
\bvolume{5}
\bpages{Paper no.\ 12, 50 pp. (electronic)}.
\bdoi{10.1214/EJP.v5-68}
\bmrnumber{1781024}
\end{barticle}
\endbibitem

\end{thebibliography}

\end{document}